\documentclass{amsart}

\usepackage{amsfonts,amssymb,amsthm,amsmath}
\usepackage{color,graphicx,ulem}
\usepackage[latin1]{inputenc}
\usepackage{dsfont}
\newcommand{\N}{\mathbb{N}}
\newcommand{\R}{\mathbb{R}}
\newcommand{\C}{\mathbb{C}}
\newcommand{\Z}{\mathbb{Z}}
\newcommand{\Gc}{\mathcal{G}}
\newcommand{\Hc}{\mathcal{H}}
\newcommand{\cH}{\mathcal{H}}
\newcommand{\Lc}{\mathcal{L}}
\newcommand{\cL}{\mathcal{L}}
\newcommand{\Kc}{\mathcal{K}}

\newcommand{\Fc}{\mathcal{F}}
\newcommand{\Vc}{\mathcal{V}}
\newcommand{\Wc}{\mathcal{W}}
\newcommand{\Sc}{\mathcal{S}}

\newcommand{\Spec}{\sigma}
\newcommand{\dom}{\operatorname{D}}
\newcommand{\disc}{\mathrm{disc}}
\newcommand{\ess}{\mathrm{ess}}
\newcommand{\diag}{\operatorname{diag}}
\newcommand{\Specess}{\sigma_{\mathrm{ess}}}
\renewcommand{\ker}{\Ker}
\newcommand{\Ker}{\operatorname{Ker}}
\newcommand{\Span}{\operatorname{Span}}
\newcommand{\ud}{\mathrm{d}}
\newcommand{\wto}{\rightharpoonup}
\newcommand\1{{\ensuremath {\mathds 1} }}
\newcommand{\dqed}{\hfill$\diamond$}
\newcommand{\rest}{\!\upharpoonright\!}
\newcommand\ii{{\ensuremath {\infty}}}
\newcommand\pscal[1]{{\ensuremath{\left\langle #1 \right\rangle}}}
\newcommand{\norm}[1]{ \left| \! \left| #1 \right| \! \right| }
\newcommand{\tr}{\mathrm{tr}}

\theoremstyle{plain}
\newtheorem{lemma}{Lemma}
\newtheorem{proposition}[lemma]{Proposition}
\newtheorem{theorem}[lemma]{Theorem}
\newtheorem{corollary}[lemma]{Corollary}

\theoremstyle{definition}
\newtheorem{example}{Example}
\newtheorem{definition}{Definition}
\newtheorem{remark}{Remark}

\title[Weyl theorems and spectral pollution]{Generalised Weyl theorems and spectral pollution in the Galerkin method}

\author[L.~Boulton]{Lyonell Boulton$^1$}
\address{$^1$Department of Mathematics and Maxwell Institute for Mathematical 
Sciences, Heriot-Watt University, Edinburgh
EH14 4AS, United Kingdom}
\email{L.Boulton@hw.ac.uk}

\author[N.~Boussaid]{Nabile Boussa{\"\i}d$^2$}
\address{$^2$D\'epartement de Math\'ematiques (CNRS-UMR 6623), UFR Sciences et techniques, 
16 route de Gray, 25 030 Besan\c{c}on 
cedex, France}
\email{nabile.boussaid@univ-fcomte.fr}

\author[M.~Lewin]{Mathieu Lewin$^3$}
\address{$^3$CNRS \& D\'epartement de Math\'ematiques (CNRS-UMR 8088),
Universit\'e de Cergy-Pontoise,
95 000 Cergy-Pontoise,
France}
\email{Mathieu.Lewin@math.cnrs.fr}

\date{November 21, 2011}

\subjclass[2000]{}

\begin{document}

\begin{abstract}
We consider a general framework for investigating spectral pollution in the Galerkin method. We show how this phenomenon is characterised via the existence of particular Weyl sequences which are singular in a suitable sense. For a semi-bounded selfadjoint operator $A$ we identify relative compactness conditions on a selfadjoint perturbation $B$ ensuring that the limiting set of spectral pollution of $A$ and $B$ coincide. Our results show that, under perturbation, this limiting set behaves in a similar fashion as the essential spectrum.
\end{abstract}
\maketitle

\bigskip

\tableofcontents


\section{Introduction}

Let $A$ be a self-adjoint operator acting on a separable infinite
dimensional Hilbert space $\mathcal{H}$ and let $\lambda$ be an isolated
eigenvalue of $A$. For $\mathcal{I}\subset \mathbb{R}$ an interval let $\1_{\mathcal{I}}(A)$ be the spectral projector of $A$ associated to $\mathcal{I}$.
The numerical estimation of $\lambda$ whenever $\inf\sigma_{\rm ess}(A)<\lambda<\sup\sigma_{\rm ess}(A)$ and, more generally, when 
\[
   \operatorname{Tr} \1_{(-\infty,\lambda)}(A)
   =\operatorname{Tr} \1_{(\lambda,\infty)}(A)=\infty,
\]
constitutes a serious challenge in applied spectral theory.
Indeed, it is well established that classical approaches, such as the Galerkin method, suffer from variational collapse under no further restrictions on the approximating space. This often leads to numerical artefacts which do not belong to the spectrum of $A$, giving rise to what is generically called \emph{spectral pollution}. 

The spectral pollution phenomenon occurs in different practical contexts such as Sturm-Liouville operators \cite{AceGheMar-06,StoWei-95,StoWei-93}, perturbations of periodic Schr{\"o}dinger operators \cite{BouLev-07,Marletta10} and systems underlying elliptic partial differential equations \cite{Arnolds, MR1705031,MR1642801}. It is a well-documented difficulty in quantum chemistry and physics, in particular regarding relativistic computations \cite{Kutzelnigg-84,StaHav-84,Grant-82,DraGol-81}. It also plays a fundamental role in elasticity and magnetohydrodynamics \cite{MR2077211,MR1935966,MR1431212,MR1285306}. 
In recent years this phenomenon has raised  a large interest in the mathematical 
community \cite{Marletta10,Hansen-08,LevSha-04,DavPlu-04,Descloux-81,Pokrzywa-81,Pokrzywa-79}. There are known pollution-free computational procedures alternative to the basic
Galerkin method. These include specialised variational formulations such as those studied at length in \cite{MR2434346,MR2077211,MR1761368,MR1406082} as well as general methods such as those proposed in \cite{BouStr-10,BouBou-09, BouLev-07,LevSha-04,DavPlu-04}.  

A natural approach to deal with spectral pollution, is to derive conditions on the approximating subspaces guaranteeing a ``safe'' Galerkin method in a given interval of the real line. These conditions were found in \cite{LewSer-09} on an abstract setting for operators with particular block-type structures with respect to  decompositions of the ambient Hilbert space. They turn out to be motivated from techniques in numerical analysis \cite{Arnolds,MR1642801,MR1431212} and computational physics and chemistry (see  references in \cite{LewSer-09}). 
 
In the present paper we adopt a more general viewpoint than that of \cite{LewSer-09}.
We establish an abstract framework for spectral pollution in the Galerkin method and then examine its invariance under relatively compact perturbations. Our main concern is primarily theoretical and general in nature. Nonetheless, however, we include various simple examples which illustrate the many subtleties faced when dealing with spectral pollution on a practical setting. 

The technical context of our results can be summarised as follows. Let $\dom(A)$ be the domain of $A$.  Let $\mathcal{L}=(\mathcal{L}_n)_{n\in \N}$ be a sequence of finite dimensional subspaces of $\dom(A)$, dense in the graph norm as $n\to\ii$ (Definition~\ref{def:A-regular})\footnote{Below we will often consider a slightly more general framework which covers
important applications such as those involving the finite element method. In this framework we will only require that the subspaces $\mathcal{L}_n$ lie in the domain of the quadratic form associated to $A$ and that the sequence
$\mathcal{L}$ is dense in the form sense. However, in this more general setting we restrict our attention to
 $A$ being  semi-bounded.}. Denote by $A_n$ the compression of $A$ to $\mathcal{L}_n$.  Denote by $\sigma(A,\cL)$ the large $n$ limiting set in Hausdorff distance of the Galerkin method spectra $\sigma(A_n)$,
(Definition~\ref{rel_spec}). Then  
$\sigma(A)\subset \sigma(A,\Lc)$ (Proposition~\ref{prop_2}), however in general equality fails to occur in this identity. An abstract notion of \emph{limiting spectral pollution set} can be formulated naturally as, 
\[
\sigma_{\operatorname{poll}}(A,\cL)=\sigma(A,\cL)\setminus \sigma(A).
\]
As it turns, points in the limiting spectral pollution set behave in a similar fashion as points in the essential spectrum (Proposition~\ref{prop:properties_of_rel_spec}). Therefore a question arises: what sort of conditions on a perturbation $B$ ensure 
$\sigma_{\operatorname{poll}}(A,\cL)=\sigma_{\operatorname{poll}}(B,\cL)$?  
Below we establish a theoretical framework in order to address this question.

Section~\ref{sec1} and \ref{sec2} are devoted to a characterisation of $\sigma(A,\cL)$ in terms of special Weyl-type sequences ($\cL$-Weyl sequences) and its structural properties. In Definition~\ref{rel_ess_spec} we consider a decomposition of $\sigma(A,\cL)$ as the union of a \emph{limiting  essential spectrum} associated with $\cL$, $\sigma_\ess(A,\cL)$, and its \emph{limiting  discrete spectrum} counterpart, $\sigma_\disc(A,\cL)$. The former contains both the true essential spectrum $\sigma_\ess(A)$ and $\sigma_{\operatorname{poll}}(A,\cL)$ (Proposition \ref{prop:quasi_Weyl}).

The purpose of sections~\ref{sec3} and \ref{sec4} is to find conditions on $B$ ensuring
\begin{equation}
\sigma_\ess(B,\cL)=\sigma_\ess(A,\cL).
\label{eq:intro_perturb} 
\end{equation}
According to our main result (Theorem~\ref{thm:Weyl}), when $A$ and $B$ are bounded from below and 
\begin{equation}
(A-a)^{1/2}(B-a)^{-1/2}-1
\label{eq:intro_condition_compact} 
\end{equation}
is a compact operator for some $a$ negative enough, \eqref{eq:intro_perturb} holds true. Therefore, an approximating sequence $\cL$ will not asymptotically pollute for $A$ in a given interval if and only if it does not pollute for $B$ in the same interval. This generalises \cite[Corollary~2.5]{LewSer-09}. 

Our present approach consists in adapting to the context of limiting spectra, several classical results for the spectrum and essential spectrum. In turns, this leads to many unexpected difficulties which we will illustrate on a variety of simple examples. In particular, we establish (Theorem \ref{thm:mapping}) a  limiting spectra version of the spectral mapping theorem allowing to replace the unbounded operator $A$ by its (bounded) resolvent $(A-a)^{-1}$. Remarkably, this theorem fails in general (Remark~\ref{remark5}) for operators which are not semi-bounded.

\section{Limiting spectra} \label{sec1}

We will often restrict our attention to $A$ being bounded from below, however we do not require this for the moment. 
Unless otherwise specified, we always assume that the subspaces $\cL_n$ are dense in the following precise sense.

\begin{definition}[$A$-regular Galerkin sequences]\label{def:A-regular}
We say that $\cL=(\Lc_n)$, $\cL_n\subset \dom(A)$, is an \emph{$A$-regular
Galerkin sequence}, or simply an $A$-regular sequence, if for all $f\in\dom(A)$ there exists a sequence of vectors $(f_n)$ with $f_n\in\cL_n$ such that $f_n\to f$ in the graph norm of $A$, that is:
\begin{equation} \label{e:con_seq}
\|f_n-f\|+\|Af_n-Af\|\to_{n\to\ii} 0.
\end{equation}
\end{definition}

The orthogonal projection in the scalar product of $\cH$ onto $\Lc_n$ will 
be denoted by 
$\pi_n:\Hc\longrightarrow \Lc_n$ and the compression of $A$ to 
$\Lc_n$ by $A_n=$ $\pi_{n}A\!\rest\!_{\Lc_{n}}:\Lc_{n}\longrightarrow \Lc_{n}$.
These compressions will sometimes be identified with any of their matrix
representations.
On sequences $(x_n)_{n\in \N}\subset \Hc$ of vectors and  $(\cL_n)_{n\in \N}$ of subspaces
$\cL_n\subset \dom(A)$ we will often suppress the index and write $(x_n)$ and
$(\cL_n)$ instead. 
we will denote by $x_n\wto x$ the fact that $x_n$ is weakly convergent to $x\in\cH$. When the norm is not specified, $x_n\to x$ will denote the fact that $\|x_n-x\|\to 0$. 

When $A$ is semi-bounded, we may also consider sequences $\Lc=(\Lc_n)$ only in the form domain of $A$. They may 
approximate the latter but not necessarily the operator domain. If $A\geq0$, for instance, this simply means that $\cL_n\subset \dom(A^{1/2})$ and $\cL$ is $A^{1/2}$-regular but not necessarily $A$-regular. In our notation, for $x_n\in \cL_n$, $Ax_n\in \dom(A^{1/2})^\#$, the dual of $\dom(A^{1/2})$ as subspace of $\mathcal{H}$. Since $\pi_n A^{1/2} y\in \mathcal{H}^\# =\mathcal{H}$ and $\pi_nA^{1/2}y\perp g$ for any $y\in \mathcal{H}$ and $g\in \mathcal{H}\ominus \cL_n$, the compression
$\pi_nA\!\upharpoonright\!_{\cL_n} \!:\!\cL_n\longrightarrow \cL_n$ is well defined also in this framework. 
Moreover, a matrix representation of
$A_n$ can be obtained in the usual manner, via $[\langle A^{1/2} b_j,A^{1/2}b_k \rangle ]_{jk=1}^{\dim \cL_n}$ for a given orthonormal basis $\{b_j\}$ of $\cL_n$. We will denote the duality product associated to $w\in \dom(A^{1/2})^\#$ by
$z\longmapsto \langle z| w\rangle$.

When $A$ is not semi-bounded but its essential spectrum has a gap containing a
number $a$, we could as well consider sequences $(\cL_n)$ which are only
$|A-a|^{1/2}$-regular. We have chosen to avoid mentioning quadratic forms for
operators which are not semi-bounded, because in practical applications (such as those involving the Dirac operator) the domain of $|A-a|^{1/2}$ does not necessarily coincide with
the natural domain upon which the quadratic form is defined. 


The limiting spectrum of $A$ relative to the Galerkin sequence 
$\cL$, is the set of all limit points, up to subsequences, of the spectra of $A_n$ in the large $n$ limit.

\begin{definition}[Limiting spectrum] \label{rel_spec}
The \emph{limiting spectrum of $A$} relative to $\Lc$, 
 $\Spec(A,\Lc)$, is the set of all $\lambda\in \R$ for which there exists  $\lambda_k\in \Spec(A_{n_k})$  such that $n_k\to\ii$ and $\lambda_k\to \lambda$ as $k\to\ii$.
\end{definition}

Since all $A_n$ are Hermitian endomorphisms,  
$\Spec(A,\Lc)\subset \mathbb{R}$.
The following lemma provides an alternative characterisation of $\Spec(A,\cL)$.

\begin{lemma}[$\cL$-Weyl sequences] \label{lem:quasi_ws}
The real number $\lambda\in \Spec(A,\Lc)$ if and only if 
there exists a  sequence $x_k\in\Lc_{n_k}$ such that $\left\|x_k\right\|=1$ and
$\pi_{n_k}\left(A-\lambda\right)x_{k}\to 0$ as $k\to\ii$. 
\end{lemma}

\begin{proof}         
According to the definition, $\lambda\in \Spec(A,\Lc)$ if and only if 
there exists $\lambda_k\in \R$ and $x_{k}\in \Lc_{n_k}$ with 
$\left\|x_{k}\right\|=1$ such that 
$\lambda_{k}\to \lambda$  and 
$\pi_{n_k}\left(A-\lambda_{k}\right)x_{k}= 0$. As $\pi_{n_k}\left(A-\lambda\right)x_{k}= (\lambda_k-\lambda)x_k\to0$, one  implication follows immediately.

On the other hand, let $(x_k)$ be as stated. Since the $A_n$ are Hermitian, there necessarily exists $\lambda_k\in \Spec(A_{{n_k}})$ such that $|\lambda_k-\lambda|\leq\norm{(A_{{n_k}}-\lambda)x_{k}}\to0$.
Thus $\lambda\in \Spec(A,\Lc)$ ensuring the complementary implication.
\end{proof}

We call $(x_{k})$ an \emph{$\cL$-Weyl sequence} for 
$\lambda\in\Spec (A,\Lc)$, by analogy to the classical notion of Weyl 
sequence \cite{Davies}.

\begin{remark} \label{1}
Selfadjointness of $A_n$  is crucial in Lemma~\ref{lem:quasi_ws}. We illustrate this by means of 
a simple example.  Let $\cH=\ell^2(\N)$ and $(e_j)\subset \cH$ be the canonical
orthonormal basis of
this space.  Let $A$ be the left shift operator defined by the condition 
$A:e_j\longmapsto e_{j-1}$ with the convention
$e_0=0$. Let $\Lc_k=\Span\left\{e_i, i\leq k\right\}$. For this data an analogous 
of Lemma~\ref{lem:quasi_ws} is no longer valid. Indeed, if $|\lambda|<1$
and
$$x_k:=\sqrt{\frac{1-|\lambda|^2}{1-|\lambda|^{2k}}}\sum_{i=1}^k \lambda^{i-1}
e_i,$$ 
then $x_k\in \Lc_k$, $\norm{x_k}=1$  and 
$$\norm{Ax_k-\lambda
x_k}=\sqrt{\frac{1-|\lambda|^2}{1-|\lambda|^{2k}}}\left|\lambda\right|^{k}
\to 0.$$ 
Therefore any point of the open unit disk is associated with an $\cL$-Weyl sequence.
On the other hand, however, $A_n$ is a Jordan block, so $\Spec(A_n)=\{0\}$ for all $n\in \N$ and hence necessarily $\Spec(A,\cL)=\{0\}$.\dqed
\end{remark}

The above characterisation of points in the limiting spectrum combined with the minimax principle yields the following fundamental statement.
\begin{proposition}[The limiting spectrum and the spectrum] \label{prop_2}
Let $\cL$ be an $A$-regular Galerkin sequence or, if $A\geq0$, an $A^{1/2}$-regular Galerkin sequence. Then, 
\begin{equation}
\label{basic_enclosure}
\Spec(A)\subset \Spec(A,\Lc)
\end{equation}
and 
\begin{equation}
\Spec_{\operatorname{poll}}(A,\cL):=\Spec(A,\Lc) \setminus \Spec(A)\subset \left(\ell^-\,,\,\ell^+\right)
\label{eq:convex_hull_essential_spectrum}
\end{equation}
where 
$$\ell^-:=\left\{\begin{array}{ll}
-\ii&\text{for $\inf\sigma(A)=-\ii$}\\
\inf \Spec_\ess(A)&\text{ otherwise}
\end{array}\right.$$
$$\ell^+:=\left\{\begin{array}{ll}
+\ii&\text{ for $\sup\sigma(A)=+\ii$}\\
\sup \Spec_\ess(A)&\text{ otherwise.}
\end{array}\right.$$
\end{proposition}

\smallskip

\begin{proof}
We start with the general case of an $A$-regular sequence. The classical characterisation of the spectrum of selfadjoint operators
ensures that $\lambda\in \Spec(A)$ if and only if there is a normalised 
sequence $(y_k)\subset \dom(A)$ such that $\|(A-\lambda)y_k\|\to 0$
(that is $(y_k)$ is a Weyl sequence for $\lambda$). We will now construct an 
$\cL$-Weyl sequence from $(y_k)$.
According to
\eqref{e:con_seq}, we can find $(x_m^{k})_{(k,m)\in\N^2}$ 
such that $x^k_m\in\Lc_{m}$, $(y_k-x^k_m)\to 0$ and $(A-\lambda)(y_k-x^k_m)\to 0$ 
as $m\to \infty$. By virtue of a diagonal process, we can extract  
a subsequence such that $\pi_{m_k}(A-\lambda)x_{m_k}^{k}\to 0$. Dividing by 
$\|x_{m_k}^k\|$ (which does not vanish in the $k\to\infty$ limit), 
gives \eqref{basic_enclosure} as consequence of Lemma~\ref{lem:quasi_ws}.

When $A\geq0$ and $\cL$ is only $A^{1/2}$-regular, a similar proof applies. We
explicitly take $x_m^{k}:=\pi'_m y_k$ where $\pi'_m$ is the
orthogonal projection onto $\cL_m$, for the scalar product furnished by the
quadratic form associated to $A$.
For all $z\in\cL_m$ with $\norm{z}=1$
\begin{align*}
\big|\pscal{z|(A-\lambda)x_m^{k}}|&= \big|\pscal{z|(A-\lambda)y_k}+(\lambda+1)\pscal{z|y_k-x_m^k}\big|\\
&\leq \norm{(A-\lambda)y_k}+(\lambda+1)\norm{y_k-x_m^k},
\end{align*}
where we have used that $\pscal{z|(A+1)y_k}=\pscal{z|(A+1)\pi'_my_k}$ by
definition of the projection $\pi'_m$. Thus
$$\norm{\pi_m(A-\lambda)x_m^{k}}\leq
\norm{(A-\lambda)y_k}+(\lambda+1)\norm{y_k-x_m^k}.$$
Our assumption that $\cL$ is $A^{1/2}$-regular implies that $x_m^k\to y_k$ in
$\dom(A^{1/2})$ when $m\to\ii$. Hence, the desired conclusion is achieved, once again, by a diagonal argument.

The proof of \eqref{eq:convex_hull_essential_spectrum} is a classical consequence of the minimax principle. It may be found, for instance, in 
\cite[Theorem 2.1]{LevSha-04} or \cite[Theorem 1.4]{LewSer-09}.
\end{proof}

 In  \cite[Theorem 1.4]{LewSer-09} the existence of an $A$-regular Galerkin sequence $\cL$ such that $\sigma(A,\cL)=[\ell^-,\ell^+]$ is shown. Therefore the inclusion complementary to \eqref{basic_enclosure} does not hold 
in general. This is a source of difficulties in applications as
there is no known systematic procedure able to identify $A$-regular 
Galerkin sequences such that $\Spec(A)=\Spec(A,\Lc)$. 
By virtue of \eqref{eq:convex_hull_essential_spectrum}, limiting spectral pollution 
$\Spec_{\mathrm{poll}}(A,\cL)$ can only occur in ``gaps'' of the essential spectrum.


Let us now see how $\Spec_{\operatorname{poll}}(A,\cL)$ can be characterised in a more precise manner in terms of particular $\cL$-Weyl sequences.

\begin{definition}[Limiting  essential spectrum] \label{rel_ess_spec}
We denote by $\Spec_{\rm ess}(A,\Lc)$ the set of all $\lambda\in\Spec(A,\cL)$ for which there exists an $\cL$-Weyl sequence $(x_k)$ as in Lemma~\ref{lem:quasi_ws} with the additional property that $x_k\wto0$.
\end{definition}

By analogy to the classical notions, we will call $\Spec_{\rm ess}(A,\Lc)$ the \emph{limiting essential spectrum of $A$} associated to $\Lc$ and the corresponding sequence $(x_k)$  a \emph{singular} $\cL$-Weyl sequence. 

\begin{remark}\label{rmk:ess_spectrum} 
From the definition it follows that
$\sigma_\ess(A+K,\cL)=\sigma_\ess(A,\cL)$ for any selfadjoint operator $K\in \Kc(\cH)$.\dqed
\end{remark}

\begin{definition}[Limiting  discrete spectrum] \label{rel_disc_spec}
The residual 
set $\Spec_{\rm disc}(A,\Lc)=\Spec(A,\Lc)\setminus \Spec_{\rm ess}(A,\Lc)$,
will be called the  \emph{limiting discrete spectrum of $A$} associated to $\Lc$. 
\end{definition}

We illustrate these definitions by means of various simple examples.

\begin{example}[$A$ a bounded operator] \label{ex0}
Let $\Hc=\Span\{e^\pm_n\}_{n\in\N}$ where $e^\pm_n$ is an orthonormal set of vectors in a given scalar product.
Let $\Lc_n=\Span\{e_1^\pm,\ldots,e_{n-1}^\pm,f_n\}$ where 
$f_n=(\cos\theta) e_n^++(\sin\theta) e_n^-$ for $\theta\in(0,\pi/2)$. Let\footnote{Here and elsewhere we use the bra-ket notation $|f\rangle\langle g|$ to denote the linear operator $\psi\mapsto \pscal{g,\psi}f$.}
\[
     A=\sum_{n\geq1}|e_n^+\rangle \langle e^+_n|,
\]
that is, $A$ is the orthogonal projector onto $\Span(e_n^+)$ and $\sigma(A)=\sigma_\text{ess}(A)=\{0,1\}$. Then $\sigma(A_n)=\{0,1,\cos^2\theta\}$ for all $n$ and $\sigma(A,\cL)=\sigma_\ess(A,\cL)=\{0,1,\cos^2\theta\}$. Here $x_n=e_n^-$ is a singular $\cL$-Weyl sequence associated to $\lambda=0$, $x_n=e_n^+$ is a singular $\cL$-Weyl sequence associated to $\lambda=1$
and $x_n=f_n$ is a singular $\cL$-Weyl sequence associated to $\lambda=\cos^2\theta$.\dqed
\end{example}

\begin{example}[$A$ a semi-bounded operator] \label{ex2}
 Let $\Hc$ be as in Example~\ref{ex0} and
define $$\Lc_n=\Span\{e_1^\pm,\ldots,e_{n-1}^\pm,e_n^-\}.$$ For
$f^\pm_n=\sin (\frac 1n)\, e_n^\mp \pm \cos (\frac 1n)\, e_n^\pm$, let
\[
     A=\sum n^2|f_n^+\rangle \langle f^+_n| - \sum |f_n^-\rangle \langle f^-_n|
\]
which has a $2\times 2$ block diagonal representation in the basis 
$(e^\pm_n)$. Then $\Spec_\ess(A)=\{-1\}$ and $\Spec_\disc(A)=\{n^2:n\in\N\}$. 
On the other hand 
\[
\Spec(A_n)=\left\{-1,n^2\sin^2\frac1n -\cos^2\frac 1n,1,\ldots,
(n-1)^2\right\},
\]
where $-1$ is an eigenvalue of multiplicity $n-1$. Therefore
\begin{gather*}
   \Spec_\ess(A,\Lc)=\{-1,0\} \qquad 
   \text{and}\qquad \Spec_\disc(A,\Lc)=\{n^2:n\in\N\}.
\end{gather*}
The former is a consequence of Proposition~\ref{prop:properties_of_rel_spec}-(ii)
while the latter follows from Proposition~\ref{prop:quasi_Weyl}-(iii) below.

We can verify directly the validity of the latter as follows. 
Assume that conversely $(x_k)$ was a singular $\cL$-Weyl sequence associated with 
$\nu^2\in\Spec_\disc(A)$. Then $\pi_{n_k}(A-\nu^2)x_k\to 0$ and $x_k\wto 0$. For $m<n_k$
\[
p_m \pi_{n_k}(A-\nu^2) x_k= (A-\nu^2)p_m  x_k
\]
where $p_m=\sum_{i\leq m} |f_i^\pm\rangle \langle f^\pm_i|$. Then, on the one hand, 
\[\|p_{n_k-1} x_k-\langle f^+_{\nu},x_k\rangle f_{\nu}^+\|^2\leq
\norm{(A-\nu^2)p_{n_k-1} x_k}^2 \to 0
\]
so that $\norm{p_{n_k-1} x_k}^2+\norm{(A-\nu^2) p_{n_k-1} x_k}^2\to 0$
as $k\to \infty$. On the other hand, 
\[(A-\nu^2)(x_k-\langle e^-_{n_k},x_k\rangle e_{n_k}^-)=(A-\nu^2)p_{n_k-1}x_k.\]
Since $\left|\langle (A-\nu^2)e_n^-,e_n^-\rangle\right|=\left|n^2\sin^2\frac1n+\cos^2\frac1n-\nu^2\right|\to \left|2-\nu^2\right|>0$, 
projecting each term onto $\Lc_{n_k}$ yields
$\langle e^-_{n_k},x_k\rangle\to 0$ also. But then $1=\norm{x_k}\to0$, which 
is a contradiction, so there are no singular $\cL$-Weyl sequences for
$\nu^2$. \dqed
\end{example}

\begin{example}[$A$ a strongly indefinite operator] \label{ex1}
Let $\Hc$ and $\Lc_n$ be as in Example~\ref{ex2}.
Let
$f^\pm_n=\frac1{\sqrt{2}}e_n^+\pm\frac1{\sqrt{2}}e_n^-$. Let
\[
     A=\sum n|f_n^+\rangle \langle f^+_n| - \sum n|f_n^-\rangle \langle f^-_n|.
\]
Then $\Spec(A)=\{\pm n:n\in\N\}=\Spec_\disc(A)$. On the other hand 
\[
  \Spec(A,\Lc)=\Z,  \qquad \Spec_\ess(A,\Lc)=\{0\}
   \qquad \text{and}\qquad \Spec_\disc(A,\Lc)=\{\pm n:n\in\N\}.
\]
The proof of the latter is similar to that of
the analogous property in Example~\ref{ex2}.
\dqed
\end{example}


\section{Limiting spectra and the behaviour of singular $\cL$-Weyl sequences}
\label{sec2}
We now examine more closely various basic properties of the limiting spectra 
$\sigma(A,\cL)$, $\sigma_{\rm ess}(A,\cL)$ and $\sigma_{\rm disc}(A,\cL)$.  These properties can be deduced
via an analysis of the behaviour of different types of $\cL$-Weyl sequences. 

\begin{proposition}[Limiting essential and discrete spectra and the spectrum]
\label{prop:properties_of_rel_spec} Let $\cL$ be an $A$-regular 
Galerkin sequence, or,  if $A\geq0$, an $A^{1/2}$-regular Galerkin sequence. Then
\begin{itemize}
\item[(i)] the limiting spectrum $\sigma(A,\cL)$ and the limiting essential spectrum $\sigma_\ess(A,\cL)$ are closed subsets of $\R$; 
\medskip

\item[(ii)] moreover $\Spec_{\rm ess}(A)\subset \Spec_{\rm ess}(A,\Lc)$ and $\Spec_{\rm disc}(A,\Lc)\subset \Spec_{\rm disc}(A)$. 
\end{itemize}
\end{proposition}

\smallskip

\begin{proof}         
The proof of (i) involves a  standard  diagonal 
argument and it is left to the reader.
For the second statement we need the following auxiliary result
which will be used repeatedly below.
\begin{lemma} \label{le:ker}
A sequence $x_k\in \Lc_{n_k}$ is such that $\|x_k\|=1$, $x_k\wto x$ and
$\pi_{n_k}(A-\lambda)x_k\to 0$, only when $x\in \Ker(A-\lambda)$. 
\end{lemma}
\begin{proof}[Proof of Lemma \ref{le:ker}]
Suppose that $(x_k)$ satisfies the hypothesis with $\cL$ an $A$-regular Galerkin sequence. Let $f\in \dom(A)$ and $f_n\in \Lc_n$ 
such that $f_n\to f$ in the norm of $\dom(A)$.  Then
$
   \langle \pi_{n_k}(A-\lambda)x_k,f_{n_k}\rangle \to 0.
$
On the other hand, since $f_k\to f$ in $\dom(A)$,
\begin{align*}
   \langle \pi_{n_k}(A-\lambda)x_k,f_{n_k}\rangle&=
   \langle x_k,(A-\lambda) f_{n_k}\rangle  \to \langle x,(A-\lambda)f\rangle.
\end{align*}
Thus $\langle x,(A-\lambda)f\rangle=0$ for all $f\in \dom(A)$, so that
$x\in \dom(A^*)=\dom(A)$ and $(A-\lambda)x=0$ as required.

Suppose now that $A\geq0$ and $\cL$ is only $A^{1/2}$-regular. The hypothesis implies that
$(x_k)$ is a bounded sequence in $\dom(A^{1/2})$. Then the proof reduces to the 
same argument, but taking this time $f_{n_k}$ in $\dom(A^{1/2})$.
\end{proof}

We now turn to the proof of (ii) in Proposition~\ref{prop:properties_of_rel_spec}. The fact that $\Spec_{\rm ess}(A)\subset \Spec_{\rm ess}(A,\Lc)$ is proved similarly to  \eqref{basic_enclosure}. It should only be noted that the $\cL$-Weyl sequence
found for $\lambda\in \Spec_{\ess}(A)$ additionally satisfies $x_{m_k}^{k}\rightharpoonup 0$.
For the inclusion $\Spec_{\rm disc}(A,\Lc)\subset \Spec_{\rm disc}(A)$ note
that, if $\lambda\in \Spec_\disc (A,\Lc)$, there exists $x_k\in \Lc_{n_k}$ such
that
$\|x_k\|=1$, $x_k\wto x\not=0$ and $\pi_{n_k}(A-\lambda)x_k\to 0$. As $\lambda\not\in \Spec_\ess(A)$ (by the previous part), then either $\lambda\in \Spec_\disc(A)$ or
$\lambda\not\in \Spec(A)$. By Lemma~\ref{le:ker}, the latter is impossible. 
\end{proof}

\begin{remark}
If $\Spec_{\rm ess}(A)= \Spec_{\rm ess}(A,\Lc)$ then automatically $\Spec_{\rm disc}(A)= \Spec_{\rm disc}(A,\Lc)$ and $\Spec(A)= \Spec(A,\Lc)$. \dqed
\end{remark}

\smallskip

We will now examine more closely singular $\cL$-Weyl sequences associated to points $\lambda\in \sigma_{\mathrm{ess}}(A,\cL)$.

\begin{proposition}[Singular $\cL$-Weyl sequences]\label{prop:quasi_Weyl}
Let $\cL$ be an $A$-regular 
Galerkin sequence, or, if $A>0$, an $A^{1/2}$-regular Galerkin sequence.
The real number $\lambda\in\sigma_\ess(A,\cL)$ if and only if
\begin{itemize}
\item[(i)] either $\lambda\not\in \Spec(A)$ and there exists $\lambda_k\to\lambda$ and 
$y_k\in \Lc_{n_k}$ such that $y_k\wto 0$ and $\pi_{n_k}(A-\lambda_k)y_k=0$;
\item[(ii)] or $\lambda\in \Spec_\ess(A)$ and there exists $\lambda_k\to\lambda$ and 
$y_k\in \Lc_{n_k}$ such that $y_k\wto 0$ and $\pi_{n_k}(A-\lambda_k)y_k=0$;
\item[(iii)] or $\lambda\in \Spec_\disc(A)$ and for any $\varepsilon>0$ 
\[
     \mathrm{Rank}\left(\1_{(\lambda-\varepsilon,\lambda+\varepsilon)}(A_n)\right)\geq
     \mathrm{Rank}\left(\1_{\{\lambda\}}(A)\right)+1
\]
for all $n$ large enough.
\end{itemize}
\end{proposition}

In cases (i) and (iii), $\lambda$ can in some sense be regarded as a point of
spectral pollution for $A$ relative to $\cL$. In case (iii), $\lambda\in
\Spec(A)$, but the multiplicity of the approximating spectrum $\sigma(A_n)$ is
too large for $n$ large, leading to the wrong spectral representation of $A$ in
the limit $n\to\ii$. In our definition of the polluted spectrum $\Spec_{\rm
poll}(A,\cL)$ in~\eqref{eq:convex_hull_essential_spectrum}, we have chosen to
require that $\lambda\notin\sigma(A)$, following~\cite{LewSer-09}. Any
$\lambda\in\sigma(A)$ satisfying (iii) could also be considered as a spurious
spectral point. However, in case (ii), the singular $\cL$-Weyl sequence
$(y_k)$ behaves like a classical singular Weyl sequence.

Only in cases (i) and (ii) the existence of a singular $\cL$-Weyl sequence $(y_k)$ consisting of exact eigenvectors of $A_{n_k}$ such that $\pi_{n_k}(A-\lambda_k)y_k=0$ and $\lambda_k\to\lambda$ is guaranteed. In case (iii) it may occur that all the eigenvectors of $A_{n_k}$ whose corresponding eigenvalue converges to $\lambda$, converge weakly  to a non-zero element of $\ker(A-\lambda)$, and that only a linear combination of these eigenvectors converges weakly to zero. This can be illustrated by means of a simple example.

\begin{example}[Spectral point satisfying Proposition~\ref{prop:quasi_Weyl}-(iii)] \label{two_definitions}
Let $\Hc=\Span\{e_0,e^\pm_n\}$ where $e_0,e^\pm_n$ form an orthonormal basis.
Let 
\[
A= \sum_{n=1}^\infty |e^+_n\rangle \langle e^+_n|
-\sum_{n=1}^\infty |e^-_n\rangle \langle e^-_n|.
\]
Then $\sigma(A)=\{-1,0,1\}$ and $\sigma_\ess(A)=\{-1,1\}$. The eigenvalue 0 has multiplicity one and associated eigenvector $e_0$.
Let
\begin{equation*}
   \Lc_{n}=\Span\{e_1^\pm,\ldots,e_{n-1}^\pm,f_n^\pm\}\quad  \text{ where}\quad 
f_n^\pm=\frac{e_0+\alpha_n^\pm\;e_n^+-\alpha_n^\mp\;e_n^-}{\sqrt{1+(\alpha_n^\pm)^2+(\alpha_n^\mp)^2}}
\end{equation*}
for
\[
    \alpha_n^\pm=\pm\sqrt{\frac{1\pm\frac{1}{n^2}}{2(1\mp\frac{1}{n^2})}}.
\]
Then $\Spec(A,\Lc)=\{0,\pm 1\}=\Spec_\ess(A,\Lc)$.
In this case $A_n$ has two eigenvalues approaching zero in the large $n$ limit, with corresponding eigenvectors $f_n^+$ and $f_n^-$. It is readily seen that
 $f_n^\pm\wto e_0/\sqrt{2}$ and so only
the difference $f_n^+-f_n^-$ tends weakly to zero.\dqed
\end{example}

\smallskip

\begin{proof}[Proof of Proposition \ref{prop:quasi_Weyl}] 
Let $\lambda\in\Specess(A,\cL)\subset\Spec(A,\cL)$. By definition of $\sigma(A,\cL)$ there exists a normalised sequence $(y_k)$ such that $\pi_{n_k}(A-\lambda_k)y_k=0$ and $\lambda_k\to\lambda$. The main question is whether one can ensure that $y_k\wto0$ weakly. Up to extraction of a subsequence, we may assume that $y_k\wto y\in \ker(A-\lambda)$ (by Lemma \ref{le:ker}).
If $\lambda\notin\sigma(A)$, then $\ker(A-\lambda)=\{0\}$ and necessarily $y=0$, thus (i) follows.

For the proof of (ii) we require the following auxiliary result.

\begin{lemma} \label{A}Let $\Vc\subset \dom(A)$ be a subspace of dimension $d>0$, with associated orthogonal projector $\pi_\Vc$. Let $\varepsilon>0$  be such that 
\[
    \|\pi_\Vc(A-\lambda)x\|\leq \varepsilon \|x\| \qquad
    \forall x\in \Vc.
\]
There exists $N>0$ and a sequence of spaces $\Wc_n\subset\cL_n$ of dimension $d$, such that for all $n\geq N$
\[
      \|\pi_{\Wc_n}(A-\lambda)y\|\leq 2\varepsilon \sqrt{d} \|y\| \qquad
\forall y\in \Wc_n.
\]
\end{lemma}

\smallskip

\begin{proof}[Proof of Lemma \ref{A}]
We firstly assume that $\cL$ is $A$-regular.
Let $(e_j)$ be a fixed orthonormal basis of $\Vc$. Then there exists $e_j^n\in\cL_n$ such that $\norm{e_j^n-e_j}_{\dom(A)}\to0$ when $n\to\ii$. The Gram matrix $G_n:=(\pscal{e^n_i,e^n_j})_{1\leq i,j\leq d}$ converges to the $d\times d$ identity matrix as $n\to\ii$, and therefore, for sufficiently large $n$, $\Wc_n:={\rm span}\{e^n_j,\ j=1,...,d\}$ has dimension $d$. Now we define an orthonormal basis for $\Wc_n$ by
$$f_j^n:=\sum_{k=1}^d(G_n^{-1/2})_{kj}\;e^n_k.$$
Since $(G_n^{-1/2})_{kj}\to\delta_{kj}$ and $e^n_j\to e_j$ in the graph norm, it is then clear that 
\[
    \|e_j-f_j^n\|_{\dom(A)}\to 0 \qquad n\to \infty.
\] 
This shows in particular that $\|(\pi_{\Wc_n}-\pi_\Vc)A\|\to0$, $\|\pi_{\Wc_n}-\pi_\Vc\|\to0$ and hence that
$\|\pi_{\Wc_n}(A-\lambda)f_j^n-\pi_\Vc(A-\lambda)e_j\|\to0$.
Let $N>0$ be such that 
\[
      \|\pi_{\Wc_n}(A-\lambda)f_j^n\|\leq 2 \varepsilon \qquad \forall n\geq N,\quad j=1,\ldots,d.
\]
For $y=\sum_{j=1}^d\hat{y}_jf_j^n \in \Wc_n$, we get
\[
     \|\pi_{\Wc_n}(A-\lambda) y\|\leq 2\varepsilon \sum_{j=1}^d|\hat{y}_j|\leq
     2\varepsilon \sqrt{d} \|y\|,
\]
which ensures the desired property.

When $A\geq0$ and $\cL$ is only $A^{1/2}$-regular, the proof is the same, using
convergence in $\dom(A^{1/2})$ and the fact that
$\|(\pi_{\Wc_n}-\pi_\Vc)A^{1/2}\|\to0$
\end{proof}

The proof of (ii) in Proposition~\ref{prop:quasi_Weyl} is achieved as follows. Assume that $\lambda\in\sigma_\ess(A)$. For all $d\in \N$ there exists a subspace $\Vc_d\subset \dom(A)$, such that $\dim \Vc_d=d^2$ and
\[
     \|(A-\lambda) y\|\leq \frac{1}{d^2}\|y\| \qquad \forall y\in \Vc_d,
\]      
see for instance \cite[Lemma~4.1.4]{Davies}.
According to Lemma~\ref{A} and an inductive argument, there is a sequence $(n_d)\subset \N$ and $d^2$-dimensional subspaces $\Wc_d\subset \Lc_{n_d}$, such that
\[
     \|\pi_{n_d} (A-\lambda) y\|\leq \frac{2}{d} \|y\| 
     \qquad
      \forall y\in \Wc_{d}.
\]
This ensures that $A_{n_d}$ has at least $d^2$ eigenvalues in the interval 
$\big[\lambda-{2}/{d},\lambda + {2}/{d}\big]$.

Let $(f_j^{n_d})_{j=1}^{d^2}\subset \Lc_{n_d}$ be an orthonormal set of $d^2$ eigenvectors of  $A_{n_d}$, with associated eigenvalues $(\lambda_j^{n_d})_{j=1}^{d^2}$
satisfying $|\lambda_j^{n_d}-\lambda|\leq {2}/{d}$. We inductively define the following singular $\cL$-Weyl sequence for $\lambda$:
$$\begin{array}{rll}
 y_1&=f_1^{n_1} \\
 y_2&=f_{\delta_2}^{n_2} & \text{ with } 1\leq \delta_2\leq 2^2\text{ such that } |\langle y_2,y_1\rangle |\leq 1/\sqrt{2} \\
 y_3&=f_{\delta_3}^{n_3} & \text{ with }1\leq \delta_3\leq 3^2\text{ such that } 
 |\langle y_3,y_j\rangle |\leq {1}/{\sqrt{3}}\text{ for }j=1,2 \\
 &\vdots & \\
 y_d&=f_{\delta_d}^{n_d} & \text{ with }1\leq \delta_d\leq d^2\text{ such that }
 |\langle y_d,y_j \rangle | \leq {1}/{\sqrt{d}}\text{ for } j=1,\ldots,d-1    .
\end{array}$$
The existence of $\delta_d$ is guaranteed by the fact that
$$
\forall k=1,...,d-1,\qquad      1=\|y_k\|^2\geq \sum_{j=1}^{d^2} |\langle y_k,f_j^{n_d}\rangle|^2.
$$
Indeed, there are at most $d$ indices $j$ in the above summation, such that $|\langle y_k,f_j^{n_d}\rangle|^2\geq 1/d$. Hence, in total, there are at most $d(d-1)$ indices $j$ such that $|\langle y_k,f_j^{n_d}\rangle|^2\geq1/d$ for at least one $k=1,...,d-1$.
Since $d(d-1)<d^2$ for $d\geq1$, we deduce that there is at least one index $j=:\delta_d$ such that $|\langle y_k,f_j^{n_d}\rangle|^2\leq1/d$ for all $k=1,...,d-1$. 
By construction $\|y_d\|=1$ and $|\pscal{y_i,y_j}|\leq 1/\sqrt{\max(i,j)}$. Thus $y_k\wto0$ as $k\to\ii$, ensuring (ii). 

Note that, conversely, if (i) or (ii) holds true, then $\lambda\in\sigma_\ess(A,\cL)$ by Definition~\ref{rel_ess_spec}.

Let us now prove that if $\lambda\in \Spec_\ess(A,\cL)\cap\Spec_\disc(A)$, then (iii) holds true. 
Let $x_k\in\Lc_{n_k}$ be a singular $\cL$-Weyl sequence: $\pi_{n_k}(A-\lambda)x_k\to0$, $\norm{x_k}=1$ and $x_k\wto0$. 
Let $\Vc=\ker(A-\lambda)\not=\{0\}$ and $d=\dim(\Vc)$. For $n$ sufficiently large
there is a space $\Wc_n\subset\cL_n$ of dimension $d$ such that for all $\varepsilon>0$, there exists
$N>0$ such that
\[
     \|\pi_n(A-\lambda)y\|\leq \varepsilon \|y\| \qquad \forall y\in \Wc_n
\]
whenever $n\geq N$.  Let $\Sc_k=\Span\{\Wc_{n_k},x_{k}\}$. Since $x_k\wto 0$
and 
$\Wc_{n_k}$ does not increase in dimension in the large $k$ limit, necessarily
$\dim(\Sc_{k})=d+1$ for all $k$ large enough. For all $\varepsilon>0$ there exists
$M>0$ such that
\[
     \|\pi_{n_k}(A-\lambda)y\|\leq \varepsilon \|y\| \qquad \forall y\in \Sc_k
\]
whenever $k\geq M$. This ensures that $\Spec(A_{n_k})\cap (\lambda-\varepsilon,\lambda+\varepsilon)$ contains at least $d+1$ points counting multiplicity and hence the
claimed conclusion is achieved.

It only remains to prove that (iii) implies $\lambda\in\sigma_\ess(A,\cL)$. 
Each individual eigenvector of $A_{n_k}$ might not converge weakly 
to $0$,  however there  is a linear combination of them that does it. We 
prove this as follows. Let $(f_j^k)_{j=1}^{d+1}$ be an orthonormal set  of $d+1$ eigenvectors
\[
       A_{n_k}f_j^k=\lambda_j^kf_j^k \qquad j=1,\ldots,d+1.
\]
Up to extraction of subsequences we may assume that $f_j^{k}\wto f_j\in\ker(A-\lambda)$ for all 
$j=1,\ldots,d+1$. If $f_j=0$ for some $j$, then the desired conclusion follows. Otherwise, 
since $\dim\ker(A-\lambda)=d$, there exist coefficients $(a_j)\in\C^{d+1}\setminus\{0\}$ such that
$\sum_{j=1}^{d+1}a_jf_j=0$.
Therefore, we may take
$$y_k:=\frac{\sum_{j=1}^{d+1}a_jf_j^{k}}{\sqrt{\sum_{j=1}^{d+1}|a_j|^2}}$$
as singular $\cL$-Weyl sequence for $\lambda$.
This completes the proof of Proposition \ref{prop:quasi_Weyl}.
\end{proof}

\section{Mapping of the limiting spectra} \label{sec3}
In this section we establish mapping theorems for the different limiting spectra. They are a natural generalisation of the analogous well-known result for $\sigma(A)$ and $\sigma_\ess(A)$ (see for example \cite[Section XIII.4]{ReeSim4}).
From now on we assume that $A$ is bounded from below
and we take $\cL$ to be an $(A-a)^{1/2}$-regular Galerkin sequence with
$a<\inf\Spec(A)$.

\begin{theorem}[Mapping of the limiting spectra] \label{thm:mapping}
Let $A$ be semi-bounded from below and let $a<\inf\sigma(A)$. Assume that $\cL$
is an $(A-a)^{1/2}$-regular Galerkin sequence. Then 
\begin{equation}
\lambda \in \Spec( A, \Lc) 
\quad\Longleftrightarrow \quad
(\lambda-a)^{-1} \in \Spec\left((A-a)^{-1}, \Gc\right)
\label{mapping:spectrum} 
\end{equation}
and
\begin{equation}
\lambda \in \Spec_\ess( A, \Lc) 
\quad\Longleftrightarrow \quad
(\lambda-a)^{-1} \in \Spec_\ess\left((A-a)^{-1}, \Gc\right)
\label{mapping:essential_spectrum} 
\end{equation}
where $\Gc=\big((A-a)^{1/2}\Lc_{n}\big)_{n\in\N}$.
\end{theorem}

\smallskip

\begin{remark} \label{remark5}
Recall that a selfadjoint operator $A$ is unbounded ($\dom(A)\subsetneq \Hc$) if and only if $0\in \Spec((A-a)^{-1})$ for one (hence for all) $a\not\in\Spec(A)$. As it turns out, $A$ is unbounded if and only if $0\in \Spec_\ess((A-a)^{-1},\Gc)$ for one (and hence all) $a<\min\Spec(A)$ and $(A-a)^{1/2}$-regular sequence $\cL$. Formally in Theorem~\ref{thm:mapping} this corresponds to the case $+\infty\in\sigma(A)$ and $(+\infty-a)^{-1}=0$.\dqed
\end{remark}

Evidently a result analogous to Theorem~\ref{thm:mapping} can be established when $A$ is semi-bounded from above. However, here $A$ is required to be semi-bounded, in order to be able to use 
a square root $(A-a)^{1/2}$ in the definition of $\Gc$, and also
for a more fundamental reason.  When $a$ is in a gap of the essential spectrum, it would be natural to expect an extension of the above result by considering, for example,
$\Gc=\big(|A-a|^{1/2}\Lc_n\big)_{n\in\N}$. The following
shows that this extension is not possible in general.

\begin{example}[Impossibility of extending Theorem~\ref{thm:mapping} for $A$
strongly indefinite] \label{ex:imposs_th7}
Let $\Hc$ be as in Example~\ref{ex2}.
Define
$\Lc_n=\Span\{e_1^\pm,\ldots,e_{n-1}^\pm,\cos(\theta_n)\,e_n^++\sin(\theta_n)\,
e_n^-\}$ with $\theta_n:=\pi/4-\lambda/(2n)$ for a fixed $\lambda\in
(0,1)$. Let
\[
     A=\sum n|e_n^+\rangle \langle e^+_n| - \sum n|e_n^-\rangle \langle e^-_n|.
\]
Then $\Spec(A)=\{\pm n:n\in\N\}=\Spec_\disc(A)$. On the other hand
\begin{equation*}
  \Spec(A,\Lc)=\Spec(A) \cup\{\lambda\},  
\qquad \Spec_\ess(A,\Lc)=\{\lambda\} \quad
    \text{and} \quad \Spec_\disc(A,\Lc)=\Spec(A).
\end{equation*}
Now
\[
     A^{-1}=\sum n^{-1}|e_n^+\rangle \langle e^+_n| - n^{-1}|e_n^-\rangle \langle e^-_n|
\]
and $\Gc=\sqrt{|A|}\cL=\cL$.
Since $A^{-1}$ is compact we have 
\begin{equation*}
 \Spec(A^{-1},\Gc)=\sigma(A^{-1})\quad \text{and} \quad\Spec_\ess(A^{-1},\Gc)=\Spec_\ess(A^{-1})=\{0\}.
\end{equation*}
Thus $\lambda\in\Spec_\ess(A,\Lc)$ whereas
$1/\lambda\not\in\Spec(A^{-1},\Gc)$.
\dqed
\end{example}

In fact the following example shows that no general extension of this theorem is possible  whenever 
$a$ lies in the convex hull of the essential spectrum, even for $A\in \mathcal{B}(\cH)$.

\begin{example}[Impossibility of extending Theorem~\ref{thm:mapping} for
$a\in\operatorname{Conv}\{\Spec_\ess(A)\}$]
Let $\Hc=L^2(-\pi,\pi)$ and $Af(x)=\operatorname{sgn}(x)f(x)$ for 
all $f\in \Hc$.
Then $\Spec(A)=\Specess(A)=\{\pm1\}$. If $\cL$ is any $A$-regular sequence, then
$\Spec(A,\cL)\subset [-1,1]$. Fixing $a=0$ yields $(A-a)^{-1}=A$. Thus
also $\Spec(A^{-1})=\{\pm1\}$ and $\Spec((A-a)^{-1},\Gc)\subset [-1,1]$, for any $A$-regular sequence
$\Gc$. If we construct a sequence $\cL=(\cL_n)$ leading to a spurious eigenvalue $\lambda\in\sigma_{\rm ess}(A,\cL)\cap(-1,1)$, we will always have $\lambda^{-1}\notin \sigma_{\rm ess}(A^{-1},\Gc)$, no matter what $\Gc$ is. We thus see that Theorem~\ref{thm:mapping} cannot be extended to include $a$ in the convex hull of the essential spectrum.\dqed
\end{example}

\smallskip

\begin{proof}[Proof of Theorem \ref{thm:mapping}]
Statement \eqref{mapping:spectrum} will follow immediately 
from the next result.

\begin{lemma}[Mapping for the spectrum of compressions]
\label{lem:SpectralMappping}
Let $A$ be semi-bounded from below, let $a<\inf\sigma(A)$ and $\Lc_n\subset \dom((A-a)^{1/2})$.
Then 
\[\lambda \in \Spec(\pi_n A\rest_{\Lc_n}) 
\quad\Longleftrightarrow \quad
(\lambda-a)^{-1} \in \Spec(p_n (A-a)^{-1}\rest_{\Gc_n})\]
where $\Gc_n=(A-a)^{1/2}\Lc_{n}$ and $p_n$ is the associated orthogonal 
projector.
\end{lemma} 
\begin{proof} Note that $\lambda \in \Spec(A_{n})$
if and only if there exists $x\in\Lc_{n}\setminus\{0\}$ such that
\[
\pi_{n} (A-a)^{1/2}\left((\lambda-a)^{-1}-(A-a)^{-1}\right)(A-a)^{1/2}x=\frac{1}{\lambda-a}\pi_{n} (A-\lambda)x=0.
\]
By fixing $y=(A-a)^{1/2}x\in \Gc_n\setminus\{0\}$, it is readily seen that
$\lambda \in \Spec(A_{n})$ if and only if
there exist $y\in\Gc_n\setminus\{0\}$ such that
\[
\left\langle (A-a)^{1/2}u,\;\left((\lambda-a)^{-1}-(A-a)^{-1}\right)y\right\rangle=0
\]
for all $u\in \Lc_{n}$. Therefore, the statement $\lambda \in \Spec(A_n)$
is equivalent to the existence of $y\in\Gc_n\setminus\{0\}$
such that $\left((\lambda-a)^{-1}-(A-a)^{-1}\right)y\perp \Gc_n$ which, in turns,
is equivalent to $p_n\left((\lambda-a)^{-1}-(A-a)^{-1}\right)y=0$.
\end{proof}

We now turn to the proof of \eqref{mapping:essential_spectrum}.
We begin by establishing an alternative characterisation of 
the limiting essential spectrum and then we formulate a stability result
for the limiting spectra with respect to compact perturbations of the regular Galerkin sequence.

\begin{lemma}[Alternative characterisation of $\sigma_\ess(A,\cL)$]\label{alt_def_rel_ess}
Let
\begin{gather*}
\Fc(A):=\left\{f(A)\ :\ f\in C_c\big(\mathbb{R}\setminus \Spec_\ess(A),\R\big)\right\}\\ \Fc^\pm(A):=\left\{f(A)\ :\ f\in C_c\big(\mathbb{R}\setminus \Spec_\ess(A),\R^\pm\big)\right\}.
\end{gather*}
Then
\begin{align}
    \Spec_\ess(A,\Lc)&=\bigcap_{B\in \Fc(A)}\Spec(A+B,\Lc)\nonumber\\
&=\bigcap_{B\in \Fc^+(A)}\Spec(A+B,\Lc)=\bigcap_{B\in \Fc^-(A)}\Spec(A+B,\Lc).
\label{eq:intersection_essential_spectrum} 
\end{align}
\end{lemma}

Here $C_c(\Omega,\R)$ denotes the set of all real-valued continuous functions
of compact support in the open set $\Omega$. Note that $\Fc(A)$ is a real vector
space and $\Fc^{\pm}(A)$ are cones, all spanned by projectors onto the
eigenspaces of $A$ associated with isolated eigenvalues of finite multiplicity.
At the end of this section it will become clear the reason why we highlight
the right hand side characterisation in 
\eqref{eq:intersection_essential_spectrum}.

\begin{proof}[Proof of Lemma \ref{alt_def_rel_ess}]
We only prove the first equality of \eqref{eq:intersection_essential_spectrum} 
as the proof of the other ones follows exactly the same pattern. 
It is well-known that
\begin{equation}
\sigma_\ess(A)=\bigcap_{B\in\Fc(A)}\sigma(A+B).
\label{eq:cap_essential} 
\end{equation}
Since all the operators in $\Fc(A)$ are of finite rank, then 
$\sigma_\ess(A+B)=\sigma_\ess(A)$ for all $B\in\Fc(A)$. Hence 
\eqref{eq:cap_essential} is equivalent to
\begin{equation}
\bigcap_{B\in\Fc(A)}\sigma_\disc(A+B)=\varnothing.
\label{eq:cap_empty}
\end{equation}

From Remark~\ref{rmk:ess_spectrum}, it follows that
$\Spec_\ess(A+B,\Lc)=\Spec_\ess(A,\Lc)$ for all $B\in \Fc(A)$.
Therefore $\Spec(A+B,\Lc)=\Spec_\disc(A+B,\Lc)\cup \Spec_\ess(A,\Lc)$. 
Moreover $\Spec_\disc(A+B,\Lc)\subseteq \Spec_\disc(A+B)$, by Proposition \ref{prop:properties_of_rel_spec}. Hence, by \eqref{eq:cap_empty},
$$\bigcap_{B\in\Fc(A)}\sigma_\disc(A+B,\cL)\subset\bigcap_{B\in\Fc(A)}\sigma_\disc(A+B)=\varnothing$$
and the result is proved.
\end{proof}

\smallskip

\begin{lemma}  \label{comp_pert_identity}
Let $T=T^\ast$ be such that $\|T\|<\infty$ and let $\Lc$ be a $T$-regular sequence. Let $K\in \Kc(\Hc)$. If
$-1\not\in \Spec(K)$, then
$$\Spec_\ess(T,\Lc)=\Spec_\ess(T,(1+K)\Lc) \quad\text{and}\quad \Spec_\disc(T,\Lc)=\Spec_\disc(T,(1+K)\Lc).$$
\end{lemma}

\smallskip

\begin{proof}[Proof of Lemma~\ref{comp_pert_identity}]
We firstly prove that
\begin{equation} \label{comp_per}
    \Spec_\ess(T,\Lc)\setminus \Spec_\disc(T)=
    \Spec_\ess(T,(1+K)\Lc)\setminus \Spec_\disc(T).
\end{equation}
Since 
\begin{equation} \label{symmetric_identity}
    \Lc=(1+K)^{-1}(1+K)\Lc=(1-K(1+K)^{-1})(1+K)\Lc,
\end{equation}
it suffices to show that the left hand side of \eqref{comp_per} is contained in the right hand side. Let $\lambda\in \Spec_\ess(T,\Lc)\setminus \Spec_\disc(T)$. If $\lambda\in \Spec_\ess(T)$, a direct application of Proposition~\ref{prop:properties_of_rel_spec}-(ii) ensures that $\lambda$ lies also in the right hand side of \eqref{comp_per}, so we can assume that $\lambda\not\in \Spec(A)$. According to Proposition~\ref{prop:quasi_Weyl}-(i), there exists $\lambda_k \to \lambda$ and $x_k\in \Lc_{n_k}$ such that $\|x_k\|=1$, $x_k\rightharpoonup 0$
and $\pi_{n_k}(T-\lambda_k)x_k=0$. For all $v_k\in\Lc_{n_k}$, and hence for all
$w_k=(1+K)v_k\in (1+K)\Lc_{n_k}$, we have
\begin{align*}
   0&=\langle (T-\lambda_k)x_k,v_k\rangle =
   \langle (1+K^*)^{-1}(T-\lambda_k)x_k,(1+K)v_k\rangle \\
   & =\langle (1+K^*)^{-1}(T-\lambda_k)x_k,w_k\rangle.
\end{align*}
Let $q_k$ be the orthogonal projection onto $(1+K)\Lc_{n_k}$.
Then \[q_k(1+K^*)^{-1}(T-\lambda_k)x_k=0.\] 
Now $(1+K^*)^{-1}=1-\tilde{K}$
where $\tilde{K}=K^*(1+K^*)^{-1}\in \Kc(\Hc)$. Hence
\[
     q_k(1-\tilde{K})(T-\lambda)x_k\to 0.
\]
But, since $\|T\|<\infty$ and $x_k\rightharpoonup 0$,
$\tilde{K}(T-\lambda)x_k\to 0$, so that also $q_k(T-\lambda)x_k\to 0$. Thus
$q_k(T-\lambda)y_k\to 0$ for $y_k=(1+K)x_k\rightharpoonup 0$. By re-normalising
$y_k$ in the obvious manner, we obtain a singular $\cL$-Weyl sequence for
$\lambda\in \Spec(T,(1+K)\Lc)$, ensuring \eqref{comp_per}.

To complete the proof of the first identity in the conclusion of the lemma,
suppose that $\lambda\in \Spec_\ess(T,\Lc)\cap \Spec_\disc(T)$.
For any $\mu\not=\lambda$ let $\tilde{T}=T+(\mu-\lambda)\1_{(\lambda-\varepsilon,\lambda+\varepsilon)}(T)$ where $\varepsilon>0$ is sufficiently small. Then 
$\lambda\in \Spec_\ess(\tilde{T},\Lc)\setminus \Spec_\disc(\tilde{T})$.
By virtue of \eqref{comp_per} and Remark~\ref{rmk:ess_spectrum}, 
$\lambda\in \Spec_\ess(\tilde{T},(1+K)\Lc)=\Spec_\ess(T,(1+K)\Lc)$ as needed.

We now show the second identity in the conclusion of the lemma. By virtue of \eqref{symmetric_identity} and
the first identity which we just proved, 
it is enough to verify 
\[\Spec_\disc(T,\Lc)\subset \Spec(T,(1+K)\Lc).\]
This, in turns, follows from Proposition~\ref{prop:properties_of_rel_spec}-(ii) and \eqref{basic_enclosure}, since 
\[\Spec_\disc(T,\Lc)\subset \Spec_\disc(T)\mbox{  and  } \Spec(T)\subset \Spec(T,(1+K)\Lc)\]
taking into account that $(1+K)\Lc$ is a $T$-regular sequence.
\end{proof}

\smallskip

We now complete the proof of Theorem~\ref{thm:mapping} by showing
\eqref{mapping:essential_spectrum}.
Let $\lambda\in \Spec_\ess(A,\Lc)$. By virtue of Lemma~\ref{alt_def_rel_ess},
this is equivalent to the statement
$$\forall B\in \Fc^+(A),\qquad \lambda \in \Spec(A+B,\Lc).$$ 
Since $B\geq 0$ and
$a<\min[\Spec(A+B)]$, according to \eqref{mapping:spectrum} 
the latter is equivalent to
$$\forall
B\in \Fc^+(A),\qquad (\lambda-a)^{-1}\in \Spec((A+B-a)^{-1},\Gc_B)$$ 
where $\Gc_B=(A+B-a)^{-1/2} \Lc$. Since $B$ has finite rank and is therefore 
compact,
Lemma~\ref{comp_pert_identity} ensures that the above in turns is equivalent to
$$\forall
B\in \Fc^+(A),\qquad (\lambda-a)^{-1}\in \Spec((A+B-a)^{-1},\Gc_0).$$
Note that $0\not\in \Spec((A+B-a)^{1/2}(A-a)^{-1/2})$ as the corresponding 
operator is an invertible function of $A$.
Now $(A+B-a)^{-1}=(A-a)^{-1}+\tilde{B}$, where $\tilde{B}=-(A-a)^{-1}B(A+B-a)^{-1}$ runs over all of $\Fc^-\left((A-a)^{-1}\right)$ as $B$ runs
over all $\Fc^+(A)$ and conversely. For the latter note that $f\in\Fc^+(A)$
if and only if $-f((\cdot -a)^{-1})\in \Fc^-((A-a)^{-1})$. Thus, once again by Lemma~\ref{alt_def_rel_ess}, $\lambda\in \Spec_\ess(A,\Lc)$ is equivalent to
\[(\lambda-a)^{-1}\in \Spec_\ess((A-a)^{-1},\Gc). \]
This completes the proof of Theorem~\ref{thm:mapping}.
\end{proof}

\begin{remark}  \label{standard_weyl}
The above proof
mimics the proof of the classical Mapping Theorem for the 
essential spectrum which can be deduced from the characterisation 
\[
    \Spec_\ess(A)=\bigcap_{B\in \Kc(\Hc)} \Spec(A+B),
\] 
see, e.g., \cite{ReeSim1}.\dqed
\end{remark}

\section{Stability properties of the limiting  essential spectrum} \label{sec4}

In this final section we present the main contribution of this paper. It strongly 
depends on the validity of Theorem \ref{thm:mapping}.

\begin{theorem}[Weyl-type stability theorem for the limiting spectra]\label{thm:Weyl}
Let $A$ and $B$ be two selfadjoint operators which are bounded below. 
Assume that for some $a<\inf\{\sigma(A),\sigma(B)\}$, 
\begin{equation} \label{domain_condition}
\dom((B-a)^{1/2})= \dom((A-a)^{1/2})
\end{equation}
 and
\begin{equation}
       (A-a)^{1/2}((B-a)^{-1/2}-(A-a)^{-1/2})\in \Kc(\Hc).
\label{eq:condition_compactness}
\end{equation}
Then
\[\Spec_{\rm ess}(A,\Lc)=\Spec_{\rm ess}(B,\Lc)\]
for all sequences $\cL=(\cL_n)$ which are simultaneously 
$(A-a)^{1/2}$-regular and $(B-a)^{1/2}$-regular.
\end{theorem}

Under Assumption~\eqref{domain_condition},~\eqref{eq:condition_compactness} is
equivalent to the same condition with the roles of $A$ and $B$ reversed:
\begin{equation} \label{other_condition}
       (B-a)^{1/2}((A-a)^{-1/2}-(B-a)^{-1/2})\in \Kc(\Hc).
 \end{equation}
 Note however that \eqref{domain_condition} and \eqref{eq:condition_compactness}
 do not imply necessarily that an $A$-regular sequence is also $B$-regular. For this
 it is enough to consider an example where $\dom(A)\not=\dom(B)$. Let $A=\partial_x^4$
 with domain \[\dom(A)=H^4(0,1)\cap\{u(0)=u(1)=0,\,u''(0)=u''(1)=0\}\subset L^2(0,1)=\cH.\]
 Let $B=\partial_x^4+|1\rangle \langle 1|(1-\partial_x^2)$ with domain
 \[\dom(B)=H^4(0,1)\cap\left\{u(0)=u(1)=0,\,u''(0)=u''(1)=\int_0^1 u(x)\ud x\right\}\not=\dom(A).\]
 Then $A^{1/2}=-\partial_x^2$ and $B^{1/2}=-\partial_x^2+|1\rangle \langle 1|$
 both with domain 
 \[
 \dom(A^{1/2})=\dom(B^{1/2})=H^2(0,1)\cap \{u(0)=u(1)=0\}.
\]
Since $A^{1/2}(B^{-1/2}-A^{-1/2})=|1\rangle \langle 1|B^{-1/2}$ is a rank-one
operator, $A$ and $B$ satisfy the hypotheses of Theorem~\ref{thm:Weyl}, but clearly $A$-regular sequences $\cL_n\subset \dom(A)\setminus \dom(B)$ are not $B$-regular.

\begin{remark} \label{form_bounded}
The KLMN theorem \cite{ReeSim4} ensures that if $B-A$ is a densely defined symmetric
 $A$-form-bounded operator with bound 
less than $1$, then \eqref{domain_condition} holds for $a$ 
sufficiently negative. \dqed
\end{remark}

The following example from \cite{LewSer-09} shows that Theorem~\ref{thm:Weyl} cannot be  easily  generalised to operators which are not semi-bounded.

\begin{example}[Relatively compact perturbations of the Dirac operator] 
Let $A=D^0$ and $B=D^0+V$ where $D^0$ denotes the free Dirac operator with unit mass \cite{Thaller} and $V\in C^\ii_c(\R^3)$ is a smooth non-negative function of
compact support. The ambient Hilbert space here is $\Hc=L^2(\R^3,\C^4)$. 
Under the additional assumption that $\sup V=\norm{V}_{L^\ii(\R^3)}<1$, it is
guaranteed that $0\notin\Spec(B)$. Furthermore it can be verified that
$$D\left(|A|^{1/2}\right)=D\left(|B|^{1/2}\right)=H^{1/2}(\R^3,\C^4)$$
and that
$$|A|^{1/2}\left(|B|^{-1/2}-|A|^{-1/2}\right)\in\Kc(\cH).$$
As a consequence of \cite[Theorem 2.7]{LewSer-09}, it is known that there exists a $B$-regular Galerkin sequence $\cL=(\cL_n)$ such that 
\begin{equation}
\sigma_\ess(B,\cL)\supset \big[0\,;\, \sup V\big].
\label{eq:ex_Dirac_poll_set} 
\end{equation}
These Galerkin spaces comprise upper and lower spinors, meaning that 
$$\cL_n=\Span\left\{\left(\begin{matrix}f^n_1\\ 0\end{matrix}\right),...,\left(\begin{matrix}f^n_{d_n}\\ 0\end{matrix}\right),\left(\begin{matrix}0\\g^n_1\end{matrix}\right),...,\left(\begin{matrix}0\\g^n_{d'_n}\end{matrix}\right) \right\}$$
for suitable $(f^n_j),\,(g^n_j)\subset L^2(\R^3,\C^2)$. This basis is known to be free of pollution if the external field $V=0$, that is 
$$
\sigma(A,\cL)=\sigma(D^0)=(-\ii,-1]\cup[1,\ii)=\Spec_\ess(A,\Lc).
$$
Hence $\Spec_\ess(A,\Lc)\neq \Spec_\ess(B,\Lc)$ so Theorem~\ref{thm:Weyl} fails
for operators which are strongly indefinite.\dqed
\end{example}

\smallskip

\begin{proof}[Proof of Theorem \ref{thm:Weyl}]
Denote by $K$ the operator on the left side of 
\eqref{eq:condition_compactness}. Then
\begin{equation}
(B-a)^{-1}-(A-a)^{-1}=(A-a)^{-1/2}K(B-a)^{-1/2}+(A-a)^{-1}K\in \mathcal{K}(\cH).
\label{eq:resolvent_compact}
\end{equation}
Let $\Gc:=(A-a)^{1/2}\cL$. According to \eqref{mapping:essential_spectrum}, 
$$\lambda \in \Spec_\ess( A, \Lc) 
\quad\Longleftrightarrow \quad
(\lambda-a)^{-1} \in \Spec_\ess\left((A-a)^{-1}, \Gc\right).$$
By Remark \ref{rmk:ess_spectrum}, 
$$\Spec_\ess\left((B-a)^{-1}, \Gc\right)=\Spec_\ess\left((A-a)^{-1}, \Gc\right).$$

Let $\Gc'=(B-a)^{1/2}\cL$. Then
$$\Gc=(A-a)^{1/2}\cL=(A-a)^{1/2}(B-a)^{-1/2}\Gc'=(1+K)\Gc'.$$
Note that $K=(A-a)^{1/2}(B-a)^{-1/2}-1$ and $-1\not\in \Spec(K)$ as
a consequence of the fact that 
$0\not\in \Spec((A-a)^{1/2}(B-a)^{-1/2})$ by~\eqref{domain_condition}. 
According to Lemma~\ref{comp_pert_identity},
\begin{equation}
\Spec_\ess\left((B-a)^{-1}, \Gc\right)=\Spec_\ess\left((B-a)^{-1}, \Gc'\right)
\label{eq:to_be_shown}
\end{equation}
The conclusion follows by applying 
Theorem~\ref{thm:mapping} again, this time to operator $B$.
\end{proof}

\begin{corollary}\label{cor:Weyl}
Let $A$ and $B$ be two bounded-below selfadjoint operators
such that \eqref{domain_condition} holds true for some 
$a<\inf\{\sigma(A),\sigma(B)\}$.
Assume that $C:=B-A$ is a densely defined symmetric operator such that 
\begin{equation}
    C \in \mathcal{B}(\dom((B-a)^{\beta}),\Hc)
\label{eq:hyp_bound_cor} 
\end{equation}
and
\begin{equation}
    (A-a)^{-\alpha}C(B-a)^{-\beta}\in \Kc(\Hc)
\label{eq:hyp_compact_cor} 
\end{equation}
for some $0\leq \alpha, \beta<1$ with $\alpha+\beta<1$. 
Then
\[\Spec_{\rm ess}(A,\Lc)=\Spec_{\rm ess}(B,\Lc)\]
for all sequences $\cL=(\cL_n)$ which are simultaneously 
$(A-a)^{1/2}$-regular and $(B-a)^{1/2}$-regular.
\end{corollary}

\begin{remark}
Let $A$ be a given bounded-below selfadjoint operator and assume that $A$ has a gap $(a,b)$ in its essential spectrum in the following precise sense,
$$\sigma_\ess(A)\cap(a,b)=\varnothing,\qquad \tr\left(\1_{(-\ii,a)}(A)\right)=\tr\left(\1_{(b,\ii)}(A)\right)=+\ii.$$
Let $\Pi:=\1_{(c,\ii)}(A)$ where $a<c<b$. Results shown in \cite{LewSer-09} ensure that, when the Galerkin spaces $\cL_n$ are compatible with the decomposition $\Hc=\Pi\cH\oplus(1-\Pi)\cH$ (i.e. when $\Pi$ and $\pi_n$ commute for all $n$), 
there is no pollution in the gap: $\sigma_\ess(A,\cL)\cap(a,b)=\varnothing$. 
According to \cite[Corollary 2.5]{LewSer-09},  when
\begin{equation}
(B-a)^{-1}C(A-a)^{-1/2}\in \mathcal{K}(\cH),
\label{eq:cond_LewSer} 
\end{equation}
then $\sigma_\ess(B,\cL)=\varnothing$ as well. 

In this respect, Theorem \ref{thm:Weyl} can be seen as a generalisation of these results. Although condition \eqref{eq:hyp_compact_cor} is stronger than \eqref{eq:cond_LewSer}, the statement guarantees that the \emph{whole} polluted spectrum will not move irrespectively of the $(A-a)^{1/2}$-regular Galerkin family $\cL$ and not only for those satisfying $[\Pi,\pi_n]=0$ for all $n$.\dqed
\end{remark}

\begin{example}[Periodic Schr{\"o}dinger operators]
Let $A=-\Delta+V_{\rm per}$ where $V_{\rm per}$ is a periodic potential with respect to some fixed lattice $\mathcal{R}\subset\R^d$ (for instance $\mathcal{R}=\Z^3$). Let $C=W(x)$ be a perturbation. Assume that
$$V_{\rm per}\in L^p_{\rm loc}(\R^d)\ \text{ where }\left\{\begin{array}{ll}
p=2 &\text{ if }d\leq3\\
p>2 &\text{ if }d=4\\
p=d/2 &\text{ if }d\geq5
\end{array}\right.$$
and that
$$W\in L^q(\R^d)\cap L^p_{\rm loc}(\R^d)+L^\ii_\epsilon(\R^d)$$
for $\max(d/2,1)<q<\ii$.
Then \eqref{eq:hyp_compact_cor} holds true for suitable $\alpha$, $\beta$ and $a$, and therefore 
\begin{equation}
\sigma_\ess\left(-\Delta+V_{\rm per}+W,\cL\right)=\sigma_\ess\left(-\Delta+V_{\rm per},\cL\right)
\label{eq:ess_spec_periodic} 
\end{equation}
for all $A$-regular Galerkin sequence $\cL$. See \cite[Section 2.3.1]{LewSer-09}.

A Galerkin sequence $\cL$ which does lead to any pollution in a given gap, can be found by localised Wannier functions, \cite{LewSer-09,CanDelLew-08b}. In practice, these functions can only be calculated numerically, so it is natural to ask what would be the polluted spectrum when they are known only approximately. According to \eqref{eq:ess_spec_periodic}, the polluted spectrum will not increase in size more than that of the unperturbed operator $-\Delta+V_{\rm per}$. \dqed
\end{example}

\begin{example}[Optimality of the constants in Corollary~\ref{cor:Weyl}] 
\label{ex3}
Let $\Hc$, $\Lc$, $e_n^\pm$ and $f_n^{\pm}$ be as in Example~\ref{ex2}. Let 
\[
    A=\sum_n n^\ell|f^+_n\rangle \langle f^+_n| +
    \sum |f^-_n\rangle \langle f^-_n|
\quad \text{and} \quad
B=\sum_n n^r|e^+_n\rangle \langle e^+_n| +
    \sum |e^-_n\rangle \langle e^-_n|.
\]
The matrix representation of $A$ and $B$ in the basis $e^\pm_n$
is made out of $2\times 2$ blocks placed along the diagonal. More precisely
$A=\diag[A_n]$, $B=\diag[B_n]$ and $C=\diag[C_n]$;
where
\[
 A_n=R_{-n}\begin{pmatrix} n^\ell & 0 \\ 0 & 1 \end{pmatrix}
R_{n}, \qquad B_n=\begin{pmatrix} n^r & 0 \\ 0 & 1 \end{pmatrix}
\]
and $C_n=A_n-B_n$ for
\[
    R_n=\begin{pmatrix}\cos\frac1n & \sin\frac1n \\ -\sin\frac1n & \cos\frac1n
\end{pmatrix}.
\]
Fix $a=0$ and let $L=A^{-\alpha}CB^{-\beta}$. The matrix representation 
of $L$ in the basis $e^\pm_k$
is $L=\diag[L_n]$ where we can calculate explicitly the entries as
\begin{align*}
     (L_n)_{11}&= -n^{-\beta r-\alpha \ell+r} \cos^2\frac1n +n^{- \beta
r}\sin^2\frac1n\\
&\qquad\qquad\qquad\qquad-n^{-r(\beta-1)}\sin^2\frac1n +n^{-\beta r -\alpha
\ell+\ell} \cos^2\frac1n \\[0,2cm]
(L_n)_{12}&=\cos\frac1n \sin\frac1n \left(n^{-\ell(\alpha-1)}-n^{-\alpha
\ell}\right) \\[0,2cm]
(L_n)_{21}&=\cos\frac1n \sin\frac1n \left(n^{-\beta r-\alpha \ell+\ell} - n^{-\beta
r-\alpha \ell+r}-n^{-\beta r} 
+ n^{-r(\beta-1)}\right) \\[0,2cm]
(L_n)_{22}&=\sin^2\frac1n \left(n^{-\ell(\alpha-1)} - n^{-\alpha \ell}\right).
\end{align*}
Therefore $L$ is compact,
given the following
\begin{equation} \label{eq:counterexample}
\begin{gathered}
   \ell=2, \quad 0<\beta,\alpha<1, \quad 0<r<2, \\
-\beta r-2\alpha+2<0, \quad \alpha>1/2, \quad \beta > 1-\frac{1}{r}.
\end{gathered}
\end{equation}
On the other hand, for $\ell=2$,
\[
\Spec_\ess(A,\Lc)=\{1,2\} \qquad \text{and} \qquad
\Spec_\ess(B,\Lc)=\{1\}.
\]
This example suggests that condition \eqref{eq:hyp_compact_cor}
in Corollary~\ref{cor:Weyl} is quasi-optimal for the stated range of $\beta$ 
and $\alpha$ as illustrated by Figure~\ref{fig:1}. Note however that in this example
\eqref{domain_condition} is only satisfied when $r=\ell$. 
\dqed
\end{example}

\begin{figure}[hth]
\centerline{\includegraphics[height=6cm]{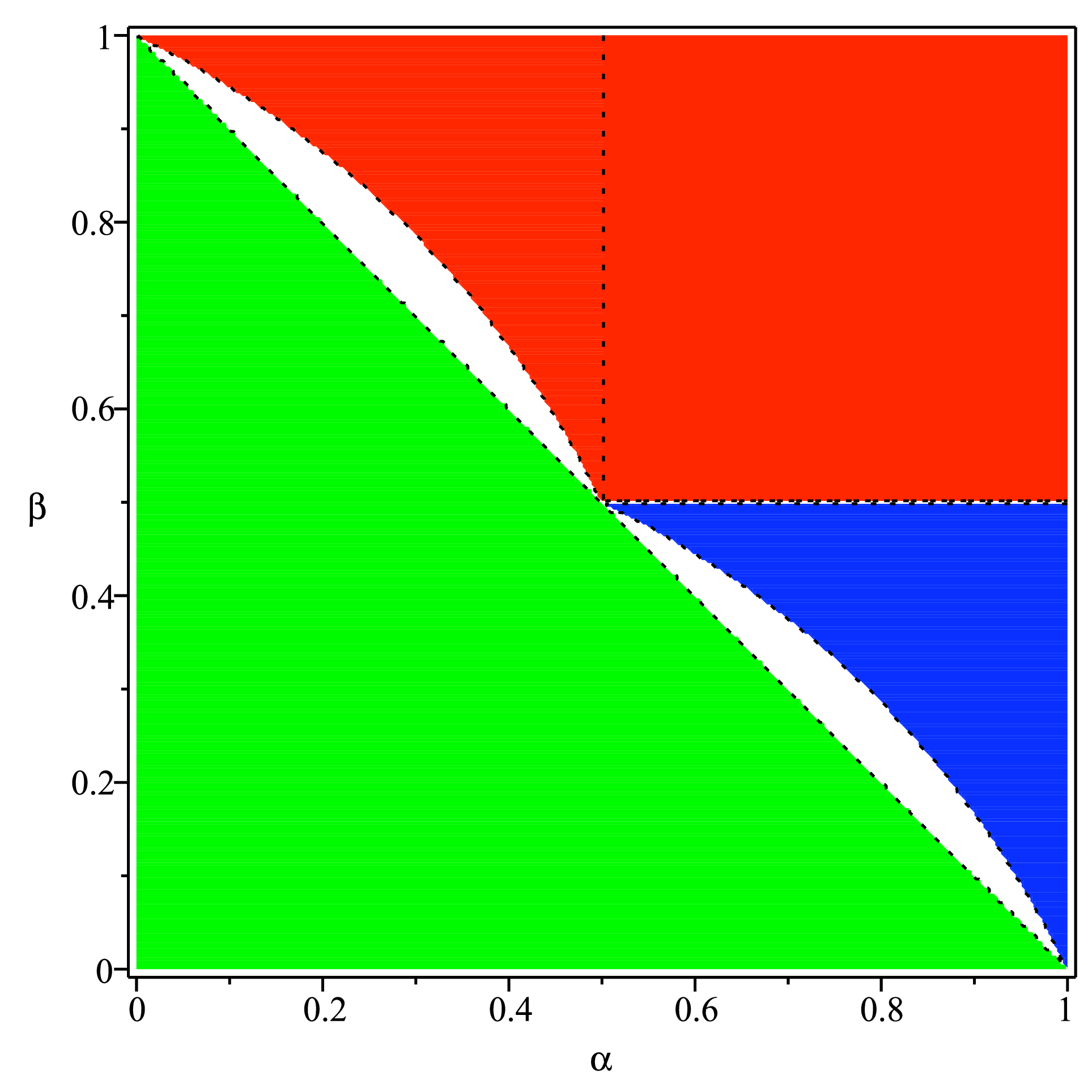}}
\caption{\label{fig:1}The region in green colour for the parameters
$\beta$ and $\alpha$ is covered by the conditions of Corollary~\ref{cor:Weyl}.
If $A$ and $B$ satisfy \eqref{eq:hyp_compact_cor} for $(\beta,\alpha)$ in this
region, 
then the limiting  essential spectrum is preserved. The region in red
shows the parameters $\beta$ and $\alpha$ in condition~\eqref{eq:counterexample}
of 
Example~\ref{ex3}. The region in blue is generated by exchanging the 
roles of $\beta$ and $\alpha$. It is not enough for $A$ and $B$ 
to satisfy \eqref{eq:hyp_compact_cor} for $(\beta,\alpha)$ in these two regions,
to 
guarantee preservation of the limiting essential spectrum.}
\end{figure}

\begin{proof}[Proof of Corollary \ref{cor:Weyl}] 
Assume firstly that $0\leq \alpha\leq 1/2$. 
The proof reduces to showing that the operator $K$ defined by
expression \eqref{eq:condition_compactness} is compact. 
Let $L=(A-a)^{-\alpha}C(B-a)^{-\beta}$ be the operator given by
\eqref{eq:hyp_compact_cor}. Since $\beta< 1$, we have
$\dom(B-a)\subset \dom(B-a)^{\beta}$, \cite[Theorem~4.3.4]{Davies}.
Then, by \eqref{eq:hyp_bound_cor}, $L\Hc\subset \dom((A-a)^\alpha)$ and
$
   Cx=(A-a)^\alpha L (B-a)^{\beta} x
$ 
for all $x\in \dom(B-a)$. By virtue of \eqref{eq:hyp_compact_cor},
\[
(A-a)^{1/2}(A-a+s)^{-1}C(B-a+s)^{-1}\in \Kc(\Hc)
\] 
for all $s\geq 0$. Moreover
\[
(A-a)^{1/2}((A-a+s)^{-1}-(B-a+s)^{-1})x=(A-a)^{1/2}(A-a+s)^{-1}C(B-a+s)^{-1}x
\] 
for all $x\in \Hc$,
as this identity is satisfied in a dense subspace of $\Hc$.
Thus
\begin{align*}
   K&=-\frac{1}{\pi}\int_0^\infty (A-a)^{1/2}(A-a+s)^{-1}C(B-a+s)^{-1} \frac{\ud s}{\sqrt{s}} \\
& = -\frac{1}{\pi}\int_0^\infty
\left\{(A-a)^{1/2}(A-a+s)^{-1}(A-a)^\alpha\right\} L 
\left\{(B-a)^{\beta}(B-a+s)^{-1}\right\} \frac{\ud s}{\sqrt{s}}.
\end{align*}
Both terms in brackets multiplying $L$ are bounded operators, 
then the integrand in the second expression is also 
a compact operator. Moreover, the integral converges in the Bochner sense
as its norm is $O(s^{\beta+\alpha-2})$ for $s\to \infty$ and
$O(s^{-1/2})$ for $s\to 0$. Thus $K\in \Kc(\Hc)$ in this case and
Theorem~\ref{thm:Weyl} implies the desired conclusion.

Now suppose that $1/2<\alpha< 1$, so that $0\leq \beta \leq 1/2$.
Since
\[
\dom (A-a)^\alpha \subset \dom(A-a)^{1/2}= \dom(B-a)^{1/2}
\subset  \dom(B-a)^{\beta},
\]
then $C\in \mathcal{B}(\dom (A-a)^\alpha,\Hc)$. Hence the operator
$(B-a)^{-\beta}C(A-a)^{-\alpha}$ is bounded and
$(B-a)^{-\beta}C(A-a)^{-\alpha}x=L^*x$ for all $x\in \Hc$. 
The proof is then completed by exchanging the roles of $A$ and $B$.
\end{proof}

\bigskip

\noindent\textbf{Acknowledgement.} The authors would like to acknowledge financial support from the British-French project PHC Alliance no. 22817YA. Partial support from the French Ministry of Research (ANR-10-BLAN-0101) is also acknowledged.
They would also like to thank \'Eric Ricard  for useful discussions.

\medskip


\end{document}